\documentclass[12pt,a4paper,reqno]{amsart}
\usepackage[T1]{fontenc}
\usepackage{lmodern}
\usepackage{amsmath,amsthm,amssymb,amsfonts,mathtools}
\usepackage[utf8]{inputenc}
\usepackage{tikz}
\usepackage{tikz-cd}
\usetikzlibrary{braids,knots}
\usepackage{enumitem}
\setlist[itemize]{label=\textbullet}

\usepackage[sorting=anyt,style=numeric,backend=biber,giveninits=true,maxnames=10,doi=false,isbn=false]{biblatex}
\DeclareFieldFormat{pages}{#1}
\renewbibmacro{in:}{}

\usepackage{listings}
\usepackage{csquotes}
\usepackage[titletoc]{appendix}
\usepackage[colorlinks,linkcolor=black,urlcolor=blue,citecolor=black]{hyperref}
\usepackage[top=2.5cm, bottom=2.5cm, left=2.5cm, right=2.5cm]{geometry}
\usepackage{cellspace}
\usepackage{adjustbox}
\usepackage{algorithm}

\usepackage[noend]{algpseudocode}
\theoremstyle{plain}
\newtheorem{thm}{Theorem}
\numberwithin{thm}{section}
\newtheorem{prop}[thm]{Proposition}
\newtheorem{cor}[thm]{Corollary}
\newtheorem{lem}[thm]{Lemma}
\newtheorem{defi}[thm]{Definition}
\newtheorem{propdef}[thm]{Proposition-Definition}
\newtheorem{conj}[thm]{Conjecture}
\newtheorem*{thm*}{Theorem}
\newtheorem*{conj*}{Conjecture}
\newtheorem*{prop*}{Proposition}

\theoremstyle{definition}

\newtheorem{ex}[thm]{Example}

\newtheorem{rmk}[thm]{Remark}

\tikzset{
    labl/.style={anchor=south, rotate=90, inner sep=.5mm}
}
\makeatletter
\def\subsection{\@startsection{subsection}{2}%
  \z@{.5\linespacing\@plus.7\linespacing}{.3\linespacing}%
  {\normalfont\bfseries}}
\makeatother

\newcommand{\Sym}{\mathfrak{S}}
\newcommand{\M}{\Sigma}
\newcommand{\pre}[2]{\prescript{#1\!\!}{}{#2}}
\newcommand{\diag}[1]{\text{diag}(#1)}
\title[Dehornoy's class and Sylows for set-theoretical solutions of the YBE]{Dehornoy's class and Sylows for set-theoretical solutions of the Yang--Baxter equation}
\author{Edouard Feingesicht}
\address{Normandie Univ, UNICAEN, CNRS, LMNO, 14000 Caen, France}
\email{edouard.feingesicht@unicaen.fr}
\makeatletter
\@namedef{subjclassname@2020}{\textup{2020} Mathematics Subject Classification}
\makeatother
\begin{document}
\begin{abstract}
We explain how the germ of the structure group of a cycle set decomposes as a product of its Sylow-subgroups, and how this process can be reversed to construct cycle sets from ones with coprime classes. We study the Dehornoy's class associated to a cycle set, and conjecture a bound that we prove in a specific case. We combine the use of braces and a monomial representation, in particular to answer a question by Dehornoy on retrieving the Garside structure without a theorem of Rump, while also retrieving said theorem.
\end{abstract}
\maketitle
\tableofcontents
\setcounter{section}{-1}
\section{Introduction}\label{intro}
In 1992 Drinfeld (\cite{drinfeld}) posed the question of classifying set-theoretical solutions of the (quantum) Yang--Baxter equation, given by pairs $(X,r)$ where $X$ is a set, $r\colon X\times X\to X\times X$ a bijection satisfying $r_1r_2r_1=r_2r_1r_2$ where $r_i$ acts on the $i$ and $i+1$ component of $X\times X\times X$. In \cite{etingof}, the authors propose to study solutions which are involutive ($r^2=\text{id}_{X\times X}$) and non-degenerate (if $r(x,y)=(\lambda_x(y),\rho_y(x))$ then for any $x\in X$, $\lambda_x$ and $\rho_x$ are bijective).  Since then, many advances have been made on this question and objects introduced: structure group (\cite{etingof}), I-structure (\cite{istruct}), etc. Many equivalent objects are known, but in particular here we are interested in cycle sets, introduced by Rump (\cite{rump}). Dehornoy (\cite{rcc}) then studied the structure group (from cycle sets) seen from a Garside perspective (divisibility, word problem, ...), he then concludes with a faithful representation, which will be the base of this article. Starting from this representation, we obtain new proofs of the Garsideness of the structure group and of the non-degeneracy of finite cycle sets, answering a question of Dehornoy in \cite{dehTalk}, and in general providing a combinatorial and algorithmical approach. Then we study finite quotient defined through an integer called the Dehornoy's class of the solution, those quotients are called germs because they come with a natural way to recover the structure monoid and its Garside structure. We state the following conjecture on Dehornoy's class (Conjecture \ref{conj1}): 
\begin{conj*} Let $S$ be a cycle set of size $n$. The Dehornoy's class $d$ of $S$ is bounded above by the ``maximum of different products of partitions of $n$ into distinct parts'' and the bound is minimal, i.e. $$d\leq\max\left(\left\{\prod\limits_{i=1}^k n_i\middle|k\in\mathbb N, 1\leq n_1<\dots<n_k, n_1+\dots+n_k=n\right\}\right).$$
\end{conj*} 
And prove it under the following conditions (Proposition \ref{conj-proof}) : 
\begin{prop*} If $S$ is square-free and its permutation group $\mathcal G$ abelian then the conjecture holds.
\end{prop*}
We then focuses on the germ and its Sylows, with the main result on cycle sets being constructed from the Zappa--Szép product of germs (Theorem  \ref{sylows}), this product being a sort of generalized semi-direct products where each term acts on the others :  
\begin{thm*}
Any finite cycle set can be constructed from the Zappa--Szép product of the germs of cycle sets of class a prime power.\end{thm*}
Taking decomposability (\cite{cycle}) into account, one can consider that the "basic" cycle sets are of class and size powers of the same prime.

The sections are organized as follows: 

Section 1 contains the introducing the different tools we will need (cycle sets, monomial matrices, braces) and answering a question by Dehornoy on retrieving the Garside structure without a theorem of Rump, while also retrieving said theorem.
Section 2 consists on the study of Dehornoy's class  and the associated germ, in particular we state a conjecture on the bound of the classes and prove it in a particular case.
Section 3 focuses on the Sylow subgroups of the germ, mainly how to construct all cycle sets from ones with coprime classes through the Zappa--Szép product of germs, with a precise condition for compatibility and an explicit algorithm/formula to do so.

\textbf{Acknowledgements} The author wishes to thank Leandro Vendramin and the referee for their insightful remarks on the content and readability of this article.
\section{Preliminaries}
The goal of this section is to provide the basic definitions of the approaches used in this article : cycle sets (\cite{rump}), braces (\cite{brace}) and Dehornoy's calculus (\cite{rcc}). Then, to use those to obtain new proofs of the Garsideness of the structure group without a theorem of Rump (answering a question of Dehornoy, but moreover also recovering Rump's theorem), and finally providing a more algorithmical approach to Dehornoy's class to be used in the next section.
\subsection{Cycle sets}
We briefly recall the basic definition that we need, that were introduced by Rump in \cite{rump}.
\begin{defi}[\cite{rump}] A cycle set is a set $S$ endowed with a binary operation $*\colon S\times S\to S$ such that for all $s$ in $S$ the map $\psi(s)\colon t\mapsto s*t$ is bijective and for all $s,t,u$ in $S$:
	\begin{equation}
	\label{RCL}
	(s*t)*(s*u)=(t*s)*(t*u).
	\end{equation}
	When $S$ is finite of size $n$, $\psi(s)$ can be identified with a permutation in $\mathfrak{S}_n$.
	
	When the diagonal map is the identity (i.e. for all $s\in S$, $s*s=s$), $S$ is called square-free.
\end{defi}
From now, we fix a cycle set $(S,*)$ .
	\begin{defi}[\cite{rump}] The group $G_S$ associated with $S$ is defined by the presentation:
	\begin{equation}
	\label{RCG}
	G_S\coloneqq\left\langle S\mid  s(s*t)=t(t*s),\: \forall s\neq t\in S\right\rangle.
	\end{equation}
	Similarly, we define the associated monoid  $M_S$ by the presentation: $$M_S\coloneqq\left\langle S\mid  s(s*t)=t(t*s),\: \forall s\neq t\in S\right\rangle^+.$$	
	They will be called the structure group (resp. monoid) of $S$.
	\end{defi}
	\begin{ex}
Let $S=\{s_1,\dots,s_n\}$, $\sigma=(12\dots n)\in\Sym_n$. The operation $s_i*s_j=s_{\sigma(j)}$ makes $S$ into a cycle set, as for all $s,t$ in $S$ we have $(s*t)*(s*s_j)=s_{\sigma^2(j)}=(t*s)*(t*s_j)$.

The structure group of $S$ then has generators $s_1,\dots,s_n$ and relations $s_is_{\sigma(j)}=s_js_{\sigma(i)}$ (which is trivial for $i=j$).

In particular, for $n=2$ we find $G=\langle s,t\mid s^2=t^2\rangle$.
\end{ex}
	When the context is clear, we will write $G$ (resp. $M$) for $G_S$ (resp. $M_S$).
	
	We also assume $S$ to be finite and fix an enumeration $S=\{s_1,\dots,s_n\}$.
	\begin{rmk} 
	By the definition of $\psi\colon S\to \Sym_n$ we have that $s_i*s_j=s_{\psi(s_i)(j)}$, which we will also write $\psi(s_i)(s_j)$ for simplicity.
	\end{rmk}	
	
\subsection{Monomial matrices} We will make use of monomial matrices. We recall the definition and some basic properties:
A matrix is said to be monomial if each row and each column has a unique non-zero coefficient. We denote by $\mathfrak{Monom}_n(R)$ the set of monomial matrices over a ring $R$. To a permutation $\sigma\in\Sym_n$ we associate the permutation matrix $P_\sigma$ where the $i$-th row contains a $1$ on the $\sigma(i)$-th column, for instance $P_{(123)}=\begin{pmatrix}0&1&0\\0&0&1\\1&0&0\end{pmatrix}$.
We then have $P_\sigma \begin{pmatrix} v_1\\\vdots\\v_n\end{pmatrix}=\begin{pmatrix} v_{\sigma(1)}\\\vdots\\v_{\sigma(n)}\end{pmatrix}$ and thus, if $e_i$ is the $i$-th canonical basis vector, $P_\sigma (e_i)=e_{\sigma^{-1}(i)}$. Moreover, for $\sigma,\tau\in\Sym_n$ we find $P_\sigma P_\tau=P_{\tau\sigma}$. It is well known that a monomial matrix admits a unique (left) decomposition as a diagonal matrix right-multiplied by a permutation matrix. Thus, if $m$ is monomial, $D_m$ will denote the associated diagonal matrix, and $P_m$ the associated permutation matrix, i.e. $m=D_mP_m$, and by $\psi(m)$ we will denote the permutation associated with the matrix $P_m$.
Let $D$ be a diagonal matrix and $P$ a permutation matrix. We denote the conjugate matrix $PDP^{-1}$ as $\pre{P}{D}$, and if $\sigma$ is the permutation associated with $P$ we will also write $\pre{\sigma}{D}$. The following statements are well-known. As they will be essential throughout this paper, we state them explicitly:
\begin{lem} 
\label{conj}
Let $D$ be a diagonal matrix and $P$ a permutation matrix. Then $\pre{P}{D}$ is diagonal.

Moreover, the $i$-th row of $D$ is sent by conjugation to the $\sigma^{-1}(i)$-th row.
\end{lem}
In particular, this implies that, $P_\sigma D=\pre{\sigma}{D}P_\sigma$ giving a way to alternate between left and right (unique) decomposition of monomial matrices.
\begin{cor} 
\label{monprod}
Let $m,m'$ be monomial matrices. Then we have $D_{mm'}=D_m\left(\pre{\psi(m)}{D_{m'}}\right)$ and $\psi(mm')=\psi(m')\circ\psi(m)$.

To simplify notations we will sometimes only write $\pre{m}{D_{m'}}$ for $\pre{\psi(m)}{D_{m'}}$.
\end{cor}
\begin{ex}
let $m=\begin{pmatrix}0&a&0\\0&0&b\\c&0&0\end{pmatrix}$, $m'=\begin{pmatrix}0&0&x\\0&y&0\\z&0&0\end{pmatrix}$, which decomposes with $D_m=\text{diag}(a,b,c)$, $D_{m'}=\text{diag}(x,y,z)$ and $\psi(m)=(123)$, $\psi(m')=(13)$. We find $\psi(m')\circ\psi(m)=(13)\circ(123)=(12)$ and $$D_m\left(\pre{\psi(m)}{D_{m'}}\right)=\diag{a,b,c}\pre{(123)}{\diag{x,y,z}}=\diag{a,b,c}\diag{y,z,x}=\diag{ay,bz,cx}$$ and indeed $mm'=\begin{pmatrix}0&ay&0\\bz&0&0\\0&0&cx\end{pmatrix}=\diag{ay,bz,cx}P_{(12)}$.
\end{ex}
\begin{defi}
We will say that a subgroup of $\M_n$ is permutation-free if it contains no permutation matrices.
\end{defi}
From \cite{etingof} we know that $G$ embeds in $\mathbb Z^S\rtimes \Sym_S$ such that projecting on the first coordinate is a bijective. We provide a matricial description of this fact :
\begin{prop}[\cite{rcc}] \label{prepr}
	Let $S$ be a cycle set of size $n$. Let $q$ be an indeterminate and consider $\mathfrak{Monom}_n(\mathbb Q[q^{\pm 1}])$ with basis indexed by $S$, denote $D_{s}$ the matrix $\text{diag}(1,\dots,q,\dots,1)$ the $n\times n$ diagonal matrix with a $q$ on the $s$-th row.
	
	The map $\Theta$ defined on $S$ by 
	\begin{equation}
	\label{repr}
	\Theta(s_i)\coloneqq D_{s} P_{\psi(s)}
	\end{equation}
	extends to a faithful representation $G\to\mathfrak{Monom}_n(\mathbb Q[q,q^{-1}]$) and similarly to a bijective morphism $M\to\mathfrak{Monom}_n(\mathbb Q[q])$.
	
	Moreover, any $f$ in $G$ is uniquely determined by $D_{\Theta(f)}$ (or $\Theta(G)$ is permutation-free), and $M$ can be identified with the submonoid $\Theta(G)\cap \mathfrak{Monom}_n(\mathbb Q[q^{\pm 1}])$ of matrices with only non-negative powers.
	\end{prop}
	From now on we identify $G$ with it's image $\Theta(G)$.
	
	Explicitly, this result means that given $(a_1,\dots,a_n)\in\mathbb Z^n$, there exists a unique element in $G$ such that for each row (resp. column) of $G$ the non-zero element is $q^{a_i}$.

Now we can state a new characterization of structure groups (braces) through the monomial matrices :
\begin{thm}
\label{condCS}
Let $G$ be a subgroup of $\M_n$, denote $G^+=G\cap \M_n^+$ (the submonoid of positive elements). Suppose that the set of atoms $S=\{s_1,\dots,s_n\}$ of $G^+$ has cardinal $n$, generates $G$ and there exists a positive integer $k$ such that $D_{s_i}=D_i^k$. Let the operation $*$ be defined on $S$ by $s_i*s_j={\psi(s_i)}(s_j)$, then the following assertions are equivalent: 
\begin{enumerate}[label=(\roman*)]
\item $G$ is permutation-free
\item $s(s*t)=t(t*s)$ for all $s,t$ in S
\item $G$ is the structure group of $S$
\end{enumerate}
\end{thm}
\begin{proof}
First notice that $q\mapsto q^p$ provides an injective morphism $\M_n\to\M_n$, so we can assume $p=1$.

(i) $\Rightarrow$ (ii): For $1\leq i,j\leq n$, we have : $$s_i s_{\psi(i)(j)}=D_iP_{s_i}D_{s_i*s_j}P_{s_i*s_j}=D_iD_jP_{s_i}P_{s_i*s_j}$$
By symmetry, $s_j(s_j*s_i)$ will have the same diagonal part. Then $$\left(s_i(s_i*s_j)\right)^{-1}\left(s_j(s_j*s_i)\right)=P_{s_i(s_i*s_j)}^{-1}D_{s_i(s_i*s_j)}^{-1}D_{s_j(s_j*s_i)}P_{s_j(s_j*s_i)}=P_{s_i(s_i*s_j)}^{-1}P_{s_j(s_j*s_i)}\in G.$$ So by the assumption that $G$ is permutation-free we deduce $s_i(s_i*s_j)=s_j(s_j*s_i)$.

(ii) $\Rightarrow$ (iii): Recall that $P_{s_i(s_i*s_j)}=P_{s_i}P_{s_i*s_j}=P_{\psi(s_i*s_j)\circ\psi(s_i)}$, so we find $\psi(s_i*s_j)\circ\psi(s_i)=\psi(s_j*s_i)\circ\psi(s_j)$. For $t\in S$, this means that $\psi(s_i*s_j)\circ\psi(s_i)(t)=\psi(s_j*s_i)\circ\psi(s_j)(t)$, i.e. $(s_i*s_j)*(s_i*t)=(s_j*s_i)*(s_j*t)$, so precisely that $S$ is a cycle set. Then the generators of $M$ correspond to the generators of $M_S$ and both are submonoids of $\M_n$, so $M=M_S$. Similarly, as $S$ generates $G$ we have $G_S=G$.

(iii) $\Rightarrow$ (i) : Comes from the previous theorem.
\end{proof}
\subsection{Braces}
Braces were first introduced by Rump in \cite{rump07} through linear cycle sets. An equivalent definition was then introduced by Ced{\'o}, Jespers and Okni{\'n}ski in \cite{leftbrace} and then in a large survey again by Ced{\'o} in \cite{brace}. We will use their definition of a (left) brace throughout this article.
\begin{defi}[\cite{rump07,brace}]
A brace is a triple $(B,+,\cdot)$ such that $(B,+)$ is an abelian group, $(B,\cdot)$ is a group and for all $a,b,c$ in $B$: $$a(b+c)+a=ab+ac.$$
$(B,+)$ will be called the additive group and $(B,\cdot)$ the multiplicative group of the brace $B$.
\end{defi}
We now fix $B$ a brace.
\begin{rmk} Note that, if 0 is the additive identity and 1 the multiplicative identity, then taking $a=1,b=c=0$ yields $1*(0+0)+1=1*0+1*0$, thus $1=0$.
\end{rmk}
\begin{ex}
If $(G,+)$ is an abelian group then $(G,+,+)$ is a brace, called the trivial brace. 

	Taking $(B,+)=\mathbb Z/2\mathbb Z\times \mathbb Z/2\mathbb Z$ with $(a,b)\cdot(c,d)=\begin{cases}(a+c,b+d),& a+b=0 \text{ mod } 2\\(a+d,b+c),& a+b=1\text{ mod } 2\end{cases}$ can be checked to be a left-brace, and obviously $(0,0)$ is the identity of $(B,\cdot)$.
\end{ex}
\begin{propdef}[\cite{brace}]
For any $a$ in $B$, the map $\lambda:(B,\cdot)\to \text{Aut}(B,+)$ defined by $\lambda_a(b)=ab-a$ for all $a,b$ in $B$, is a well-defined  morphism.

This also gives $ab=a+\lambda_a(b)$. This will be used everywhere to switch between products and sum of elements.
\end{propdef}
\begin{ex} From the previous example we have respectively $\lambda_g=\text{id}_G$ for all $g$ in $G$, and in $(B,+,\cdot)$ $\lambda((a,b))=\sigma^{a+b}$ where $\sigma$ permutes the two coordinate of $(B,+)$, and obviously $(0,0)$ is the identity of $(B,\cdot)$.
\end{ex}
\begin{lem}[\cite{brace}]\label{prodsum}\label{ab-1}\label{ab-1soc}
For any $a,b$ in $B$ we have :
\begin{enumerate}
\item $\lambda_a\lambda_b=\lambda_{a+\lambda_a(b)}$.
\item $ab^{-1}=-\lambda_{ab^{-1}}(b)+a$
\item If $\lambda_a=\lambda_b$ then $ab^{-1}=a-b$
\end{enumerate}
\end{lem}
\begin{proof}
This first one follows from $gh=g+\lambda_g(h)$.

For the second one, $-\lambda_{ab^{-1}}(b)+a=-ab^{-1}b-ab^{-1}+a=ab^{-1}$.

And then, $\lambda_{ab^{-1}}=\lambda_a\lambda_b^{-1}=\lambda_a\lambda_a^{-1}=\text{id}_B.$
\end{proof}
\begin{lem}\label{ybe}
For any $a,b$ in $B$, we have $a\lambda_a^{-1}(b)=b\lambda_b^{-1}(a)$.

Moreover, $\lambda^{-1}_{\lambda^{-1}_a(b)}\lambda_a^{-1}=\lambda^{-1}_{\lambda^{-1}_b(a)}\lambda_b^{-1}$.
\end{lem}
\begin{proof}
Firstly,
$$a\lambda_a^{-1}(b)=a(a^{-1}b-a^{-1})=b-1+a=b-0+a=b+a=a+b=b\lambda_b^{-1}(a).$$
Then from the fact that $\lambda\colon(B,\cdot)\to\text{Aut}(B,+)$ is a morphism we have that $\lambda_{ab}^{-1}=\lambda_b^{-1}\lambda_a^{-1}$ so $$\lambda^{-1}_{\lambda^{-1}_a(b)}\lambda_a^{-1}=\lambda^{-1}_{a\lambda^{-1}_a(b)}=\lambda^{-1}_{b\lambda^{-1}_b(a)}=\lambda^{-1}_{\lambda^{-1}_b(a)}\lambda_b^{-1}.$$
\end{proof}
The following is implicit in \cite{brace} : 
\begin{lem}[\cite{brace}]
Let $S$ be a subset of a brace $(B,+,\cdot$ such that $\lambda_s(S)\subseteq S$ for any $s$ in $S$. Then $(S,+)$ is a subgroup of $(B,+)$ if and only if it is a subgroup of $(B,\cdot)$.
\end{lem}
\begin{proof}
This follows from the identity $ab=a+\lambda_a(b)$ (or equivalently $a+b=a\lambda_a^{-1}(b)$).
\end{proof}
\begin{defi}[\cite{brace}] Let $(B,+,\cdot)$ be a brace.
\begin{itemize}
\item $S\subseteq B$ is a subbrace if it is a subgroup of both $(B,+)$ and $(B,\cdot)$.
\item $L\subseteq B$ is a left ideal if it is a subgroup of $(B,+)$ and $\lambda_a(I)\subseteq I$ for all $a$ in $B$.
\item $I\subseteq B$ is an ideal if it is a normal subgroup of $(B,\cdot)$ and $\lambda_a(I)\subseteq I$ for all $a$ in $B$.
\end{itemize}
\end{defi}
\begin{prop}[\cite{brace}] Let $(B,+,\cdot)$ be a brace and $I\subseteq B$.
\begin{itemize}
\item $I$ is an ideal $\Rightarrow$ $I$ is a left ideal $\Rightarrow$ $I$ is a subbrace.
\item If $I$ is an ideal then the multiplicative quotient $B/I$ has an induced brace structure $(B/I,+,\cdot)$.
\item $\text{Soc}(B)=\text{Ker}(\lambda)=\{a\in B\mid \forall b\in B, ab=a+b\}$ is an ideal called the Socle of B.
\end{itemize}
\end{prop}
\begin{rmk}
\end{rmk}
In \cite{etingof,rcc} it is shown that the structure group $G$ of a finite cycle set $S$ of size $n$ has an I-structure : $G$ embeds in $\mathbb Z^n\rtimes\mathfrak S_n$ such that projecting on the first coordinate is a bijection. Moreover, the structure monoid $M$ embeds in $G$ and corresponds to first coordinates in $\mathbb N^n$. This is used in \cite{brace} to define the additive structure on $G$ on the first coordinate, we give an equivalent statement in terms of matrices :
\begin{thm}[\cite{brace}]\label{brace}
The structure group $G$ of a finite cycle set $S$ has a brace structure with the usual multiplication, addition given by $g+h$ as the unique element with $D_{g+h}=D_gD_h$, and such that $(G,+)\simeq \mathbb Z^S$. 

 In particular, if $g$ and $h$ are in $M$, so is $g+h$.
 
 Moreover, $\text{Soc}(G)=\{g\in G\mid P_g=\text{Id}\}$.
\end{thm}
\begin{rmk}
Note that for any $g$ in $G$, $\text{Id}_n=D_{g-g}=D_{g}D_{-g}$ thus $$D_{-g}=D_g^{-1},$$ while $0=g^{-1}g=g^{-1}+\lambda_g^{-1}(g)$ thus $$D_{g^{-1}}=D_{\psi(g)(g)}^{-1}=\pre{g}{D_g^{-1}}.$$
\end{rmk}
\begin{cor}
For any $s,t$ in $S$ we have $\lambda_s(t)=\psi(s)^{-1}(t)$, or equivalently $\lambda_s^{-1}(t)=s*t$.
\end{cor}
\begin{proof}
$\lambda_s(t)=st-s=D_sP_sD_tP_t-D_sP_s=D_sD_{\psi(s)^{-1}(t)}P_sP_t-D_sP_s$ which, from the previous theorem, is the unique element with diagonal part $(D_sD_{\psi(s)^{-1}(t)})D_s^{-1}=D_{\psi(s)^{-1}(t)}$ (as $D_{-s}=D_s^{-1}$, thus $\lambda_s(t)=\psi(s)^{-1}(t)$.
\end{proof}
From the I-structure mentioned above we can write any element of $G$ as $g=\sum\limits_{s\in S} g_s s$ where $g_s\in\mathbb Z$. 

Then for any $h$ in $G$, we have $\lambda_h(g)=\sum_S g_s \lambda_h(s)$ with $\lambda_h(s)$ in $S$.
\subsection{Garsideness}
In \cite{rcc}, Dehornoy used Rump's result on the non-degeneracy of finite cycle sets to obtain the Garside structure of the structure group (first proved in \cite{chouraqui} and also appearing in \cite{rightl}). In \cite{dehTalk} he asked whether the opposite could be done. The objective of the next subsections is to provide a positive answer this question, that is, to obtain the Garside structure without Rump's result, and then recover Rump's result (without even using the Garside structure). This subsection will mostly use the monomial matrix approach, but in the next one we will give a brace equivalent of some statements (but requiring the use a consequence of Rump's theorem).

Recall that we fix $(S,*)$ a finite cycle set of size $n$ with structure brace $G$ and monoid of positive elements $M$. Moreover, as the defining relations of the presentation of $G$ are homogeneous (quadratic), we have a well-defined length function $\ell : G\to \mathbb Z$, which restricts to $M\to \mathbb N$.

\begin{defi} Let $g_1,g_2$ be elements of $M$. We say that $g_1$ left-divides (resp. right-divides) $g_2$, that we note $g_1\preceq g_2$ (resp. $g_1\preceq_r g_2$) if there exists some $h\in M$ such that $g_2=g_1h$ (resp. $g_2=hg_1$) and $\ell(g_2)=\ell(g_1)+\ell(h)$.

An element $g\in M$ is called balanced if the set of its left-divisors $\text{Div}(g)$ and the set of its right-divisors $\text{Div}_r(g)$ coincide.
\end{defi}
Note that, as $g_1=P_{g_1}\pre{g_1}{D_{g_1}}$, its matricial transpose is given by $g_1^t=P_{g_1}^tD_{g_1}=P_{g_1}^{-1}D_{g_1}=\pre{g_1}D_{g_1}P_{g_1}^{-1}$, thus the coefficient on the $i$-th column of $g_1$ is the coefficient on the $i$-th row of $g_1^t$.
\begin{prop}\label{div} Let $g,h$ be in $M$.

Then $g$ left-divides $h$ if and only if for each row of $g$ the power of $q$ on this row is smaller than the corresponding one of $h$.

Similarly, $g$ right-divides $h$ if and only if for each column of $g$ the power of $q$ on this column is smaller than the corresponding one of $h$.
\end{prop}
We will give an alternative brace proof and an interpretation of this statement in subsection, but it will use Rump's theorem. \ref{nd}.
\begin{proof} 
Write $g_i=D_{g_i}P_{g_i}=P_{g_i}\pre{g_i}{D_{g_i}}$. For left-divisibility, consider in $G$ the element $h=g_1^{-1}g_2=P_{g_1}^{-1}D_{g_1}^{-1}D_{g_2}P_{g_2}$. Recall that $h\in M$ iff $D_{g_1}^{-1}D_{g_2}$ contains only non-negative powers of $q$ (to lie in $\mathbb N^n\subseteq \mathbb Z^n$ the additive group of the brace), precisely meaning that the power on each row of $g_1$ is less than the one of $g_2$.

Similarly, for right divisibility, let $h'=g_2g_1^{-1}=P_{g_2}\pre{g_2}{D_{g_2}}\left(\pre{g_1}{D_{g_1}}\right)^{-1}P_{g_1}^{-1}$, which is in $M$ iff $\pre{g_2}{D_{g_2}}\left(\pre{g_1}{D_{g_1}}\right)^{-1}$ contains only non-negative powers of $q$, which is the same criterion on the columns.
\end{proof}
\begin{ex}
Taking $S=\{s_1,s_2\}$ with $\psi(s_1)=\psi(s_2)=(12)$, we can see that: 

$\begin{pmatrix}0&q^3\\1&0\end{pmatrix}$ left-divides $\begin{pmatrix}q^4&0\\0&1\end{pmatrix}$ (as $3\leq 4$ on the first line and $0\leq 0$ on the second since $1=q^0$), but doesn't right divide it (as $3>0$ on the second column)
\end{ex}
\begin{cor}\label{gcdlcm}
Let $g,h$ be in $M$. The left-gcd (resp. left-lcm) of $g$ and $h$, denoted $g_1\wedge g_2$ (resp. $g_1\vee g_2$) is given by the unique element such that the coefficient-power on each row is the maximum (resp. minimum) of those of $g_1$ and $g_2$.

For right-gcd (resp. right-lcm) it is the same but for each column.
\end{cor}
\begin{ex}
Consider $S=\{s_1,s_2,s_3,s_4\}$ with \begin{align*}\psi(s_1)=(1234)&\qquad\psi(s_3)=(24)\\\psi(s_2)=(1432)&\qquad\psi(s_4)=(13)\end{align*}
We have $$\begin{pmatrix}0&q&0&0\\1&0&0&0\\0&0&0&q\\0&0&1&0\end{pmatrix}\wedge\begin{pmatrix}0&0&0&q\\0&0&q&0\\0&1&0&0\\1&0&0&0\end{pmatrix}=\begin{pmatrix}0&q&0&0\\0&0&1&0\\0&0&0&1\\1&0&0&0\end{pmatrix}$$

Which is given by $\text{gcd}\left(s_1+s_3,s_1+s_2\right)=s_1$.

Similarly: 
$$\begin{pmatrix}0&q&0&0\\1&0&0&0\\0&0&0&q\\0&0&1&0\end{pmatrix}\vee\begin{pmatrix}0&1&0&0\\q&0&0&0\\0&0&0&1\\0&0&q&0\end{pmatrix}=\begin{pmatrix}q&0&0&0\\0&q&0&0\\0&0&q&0\\0&0&0&q\end{pmatrix}$$
Which is given by $\text{lcm}\left(s_1+s_3,s_2+s_4\right)=s_1+s_2+s_3+s_4$.

For the right gcd and lcm, the explicit versions will be given in the next subsection with Rump's result.
\end{ex}
\begin{cor} An element such that the non-zero terms of its $i$-th row and $i$-th column are equal for all $1\leq i\leq n$ is balanced.
\end{cor}
\begin{defi} An element of $M$ is called a Garside element if it is balanced, $\text{Div}(g)$ is finite and generates $M$.
\end{defi}
\begin{prop}[\cite{rcc}] 
\label{gars-ele}
The element $\Delta=\sum_S s$  is a Garside element of $M$.
\end{prop}
\begin{proof} Because all the non-zero coefficients of $\Delta$ are equal, it is balanced.

Its set of divisors is the set of elements with non-zero coefficients $1$ or $q$ and so is finite and has cardinal $2^n$, and it contains all the generators $s$ so also generates M. 
\end{proof}
\begin{rmk} The powers of $\Delta$, which are given by $\Delta^k=\sum_S ks$, are also Garside elements by the same reasoning.

More generally, Garside elements are precisely the balanced elements $g$ such that $g_s\geq 1$.
\end{rmk}
\begin{defi}[\cite{garside}] A monoid is said to be a Garside monoid if: 
\begin{enumerate}[label=(\roman*)]
\item It is cancellative, i.e. if for every element $g_1,g_2,h,k$, $hg_1k=hg_2k\Rightarrow g_1=g_2$.
\item There exists a map $\Lambda$ to the integers such that $\ell(g_1g_2)\geq\ell(g_1)+\ell(g_2)$ and $\ell(g)=0\Rightarrow g=1$.
\item Any two elements have a gcd and lcm relative to $\preceq$ (resp. $\preceq_r$).
\item It possesses a Garside element $\Delta$.
\end{enumerate}
\end{defi}
\begin{prop}[\cite{rcc}] M is a Garside monoid.
\end{prop}
\begin{proof} The length of words (as the relations of $G$ respects length) satisfies (ii) as an equality.

For (iii) we have Corollary \ref{gcdlcm} and for (iv) Proposition \ref{gars-ele}.

We are left to prove (i), which is a direct consequence of the fact that $M\hookrightarrow G$. We can see also this from restricting the representation to $M$ and as the elements are monomial matrices, we deduce the cancellative property).
\end{proof}
\subsection{Dehornoy's class} 
The goal here is to follow the construction of the germ from \cite{rcc} mixing a brace and monomial approach.

Fix a finite cycle set $(S,*)$ of size $n$ with structure monoid (resp. group) $M$ (resp. $G$). Recall that the socle of $G$ corresponds to the elements with trivial permutation, in particular $(\text{Soc}(G),+)=(\text{Soc}(G),\cdot)$.
\begin{prop} \label{class} There exists a positive integer $d$ such that for all $s$ in $S$, $ds$ is diagonal, i.e. $ds\in \text{Soc}(G)$.
\end{prop}
The smallest positive integer satisfying this condition is called the Dehornoy's class of $S$, and all the others will be multiples of this class. Our results are stated for the class, but most would work for any multiples. In \cite{rcc} the elements $ks$ are denoted $s^{[k]}$.
\begin{proof}
First fix $s\in S$. The map sending $ks$ to $\psi(ks)$ is a map from an infinite (countable) set to a finite one ($\mathbb N\to\mathfrak S_n$), therefore it is not injective. So there exists $k_1,k_2\in\mathbb N$ such that $\lambda_{k_1s}=\lambda_{k_2s}$. We can assume $k_1>k_2$ without loss of generality. Then from Lemma \ref{ab-1soc} $(k_1s)\cdot(k_2s)^{-1}=k_1s-k_2s=(k_1-k_2)s\in\text{Soc}(G)$

Doing this for all $s\in S$, we get the existence of $d_s\in\mathbb N$ such that $d_ss$ is diagonal. Notice that as Soc$(G)$ is an ideal, we must have for all $k\in\mathbb N$ $k(d_ss)\in\text{Soc}(G)$. Taking $d=\text{lcm}(d_s)_{s\in S}$ we have for all $s$ the existence of $d_s'>0$ such that $d=d_sd_s'$, we find that for all $s$, $ds=d_s'(d_ss)\in\text{Soc}(G)$.
\end{proof}
\begin{rmk}
In \cite{rcc}, the author obtained a bound on the class of a cycle set as $d\leq (n^2)!$. Here, we obtain a first better bound $d\leq (n!)^n$ given by the previous proof (as $d=\text{lcm}(d_1,\dots,d_n)$ with $d_i\leq n!$). This bound will be improved in Propositions \ref{ddivG} and \ref{conj-proof}.
\end{rmk}
\begin{prop} Let $d$ be the class of $S$ and denote by $dG$ the subgroup of $G$ generated by all the $ds$ (additively or multiplicatively, as they are the same). Then $dG$ is an ideal of $G$.
\end{prop}
\begin{proof} This follows directly from the fact $dG$ is a subgroup of $\text{Soc}(G)$ : as $(\text{Soc}(G),\cdot)=(\text{Soc}(G),+)$ which is abelian, $dG$ is normal, and by definition $\lambda_{ds}=\text{id}$ so $\lambda_{h}=\text{id}$ for any $h\in dG$.
\end{proof}
Thus we obtain a quotient brace $\overline G$ by $\overline G=G/dG$.
 \begin{prop}[\cite{rcc}]
 A presentation of $\overline G$ can be obtained by adding to the presentation of $G$ the relations $ds=1$. Matricially, quotienting is the same as specializing at $q=\exp(\frac{2i\pi}{d})$, which we will denote $\text{ev}_q$.
 
 Moreover, the quotient brace $\overline G$ has additive group $(\mathbb Z/d\mathbb Z)^S$
 \end{prop}
\begin{proof}
The first part comes from the fact that $dG$ is generated by the $ds$ which are in the socle, so they have trivial permutation and quotienting by them just amounts to setting $q^d=1$.

The second part then directly follows from the fact that $(G,+)\simeq \mathbb Z^S$.
\end{proof}
\begin{rmk}
If $d=1$ then $dG=G$ so $\overline G$ is trivial. However, $d=1$ means that all the generators $s$ are diagonal, i.e. $s*t=t$ for all $s,t$ in $S$: this is just the special case of the trivial cycle set. But the case $d=1$ can be included as all our results hold for any multiples of the class (thus any positive integer for $d=1$).
\end{rmk}
From now on, assume $d\geq 2$.
\begin{ex}
Let $S=\{s_1,\dots,s_n\}$ with $\psi(s_i)=(12\dots n)=\sigma$ for all $i$. Then for any $s\in S$, $k\in\mathbb Z$: $ks_i=D_s^k P_{\sigma^k}$. Thus Dehornoy's class of $S$ is equal to $n$. Let $\zeta_n=\exp(\frac{2i\pi}{n})$, then $\overline G$ is generated by the $\overline{s_i}=\diag{1,\dots,\zeta_n,\dots,1}P_\sigma$.
\end{ex}
Denote by $\zeta_d=\exp(\frac{2i\pi}{d})$ a primitive $d$-th root of unity and $\mu_d=\{\zeta_d^i\mid 0\leq i<d\}$. Let $\M_n^d$ be the subgroup of $\mathfrak{Monom}_n(\mathbb C)$ with non-zero coefficients in $\{0\}\cup\mu_d$. Given $k\geq 1$, there is natural embedding $\iota_d^{dk}\colon \M_n^d\to\M_n^{dk}$ sending $\zeta_d$ to $\zeta_{dk}^k$ (as $\zeta_{dk}^k=\exp(\frac{2ik\pi}{dk})=\zeta_d$). From the previous proposition, we deduce the following result: 
\begin{cor}
The quotient group $\overline G$ is a subgroup of $\M_n^d$.
\end{cor}
Recall that if $S$ has Dehornoy's class $d$, then for any positive integer $k$ we have that $kds$ is in the Socle, thus we could also consider the germ $G/\langle kds\rangle_{s\in S}\rangle$. The embedding $\iota_d^{dk}(\overline G)$ can then be seen as embedding the germ $\overline G$ in this bigger quotient group.
\begin{defi}\cite{rcc} If $(M,\Delta)$ is a Garside monoid with atom set $S$ and $G$ is the group of fractions of $M$, a group $\overline G$ with a  surjective morphism $\pi\colon G\to\overline G$ is said to provide a Garside germ for $(G,M,\Delta)$ if there exists a map $\chi\colon\overline G\to M$ such that $\pi\circ\chi=\text{Id}_{\overline G}$, $\chi(\overline G)=\text{Div}(\Delta)$ and $M$ admits the presentation $$\langle \chi(\overline G)\mid \chi(fg)=\chi(f)\chi(g) \text{\:when\:} ||fg||_{\overline S}=||f||_{\overline S}+||g||_{\overline S}\rangle$$
where $||\cdot||_{\overline S}$ denote the length of an element over $\overline S=\pi(S)$.
\end{defi}
\begin{prop}[\cite{rcc}] The specialization $\text{ev}_q$ that imposes $q=\exp(\frac{2i\pi}{d})$ provides a Garside germ of $(G,M,\Delta^{d-1})$.
\end{prop}
\begin{proof} Consider the map $\chi\colon \overline G\to M$ defined by sending $\exp(\frac{2i\pi\cdot k}{d})$ to $q^k\in\mathbb Q[q]$ for $1\leq k<d$. It trivially satisfies $\text{ev}_q\circ\chi=\text{Id}_{\overline G}$. Its image is the elements of $M$ such that each non-zero coefficient is a power of $q$ strictly less than $d$, and thus identifies with $\text{Div}(\Delta^{d-1})$ by the characterization of divisibility. And the presentation amounts to forgetting that $q$ is a root of unity, thus generating $M$ as required.
\end{proof}
To work over $\overline G$, we will use the following corollary to restrict to classes of equivalence over the structure monoid.
\begin{cor}
The projection $\text{ev}_q\colon M\to\overline G$ is surjective.
\end{cor}
\begin{proof}
$G$ is generated by all the elements of $S$, so its quotient $\overline G$ is also generated by $S$. Moreover, as $\overline G$ is finite, inverses can be constructed from only positive generators, thus $G\to\overline G$ factors through $M$.
\end{proof}
\begin{ex}
Let $S=\{s_1,\dots,s_n\}$ with $\psi(s_i)=(12\dots n)=\sigma$ for all $i$. Then for any $s\in S$, $k\in\mathbb Z$: $ks_i=D_s^k P_{\sigma^k}$. The Dehornoy's class of $S$ is $n$ and $\overline G$ is generated by the $\overline{s_i}=\diag{1,\dots,\zeta_n,\dots,1}P_\sigma$ where $\zeta_n=\exp(\frac{2i\pi}{n})$.

To recover $G$ from $\overline G$, one simply takes all the elements of $\overline G$ and forget that $q$ is a root of unity in the following sense: when computing the product of two elements and finding a coefficient $q^a$ with $a>d$, we do not use that $q^d=1$ and just consider it as a new element. So for instance in $\langle \chi(\overline G)\rangle$ with $n=4$: $$\chi\!\left(\overline{3s_1}\right)\!\chi\!\left(\overline{2s_4}\right)=\chi(3\overline{s_1})\chi(2\overline{s_4})=\begin{pmatrix}0&0&0&q^3\\1&0&0&0\\0&1&0&0\\0&0&1&0\end{pmatrix}\begin{pmatrix}0&0&1&0\\0&0&0&1\\1&0&0&0\\0&q^2&0&0\end{pmatrix}\! =\!\begin{pmatrix}0&q^5&0&0\\0&0&1&0\\0&0&0&1\\0&0&0&0\end{pmatrix}=5s_1.$$ Because $5>4$, we obtain a new element different from $\chi\left(\overline{5s_1}\right)=\chi\left(\overline{s_1}\right)$.
\end{ex}
The quotient group $\overline G$ defined above is called a Coxeter-like group, it was first studied by Chouraqui and Godelle in \cite{godelle} for $d=2$ and generalized by Dehornoy in \cite{rcc}. 

Fix $\overline G$ a Coxeter-like group obtained from a cycle set $S$ of cardinal $n$ and class $d\geq 2$ (so that $\overline G$ is not trivial).

\begin{defi}
We define a function $\text{l}_d\colon\{0,1,\dots,d-1\}\to \{0,1,\dots,\lfloor\frac{d}{2}\rfloor\} $ by:
	\begin{equation}
	\label{lq}
	\forall k\in\{0,1,\dots,d-1\},\text{l}_d(k)=\begin{cases}k,&\: \text{if } k\leq\frac{d}{2}\\k-d,&\:\text{if } k>\frac{d}{2}.\end{cases}
	\end{equation}
	And we define $\overline\ell(c)=\sum\limits_{i=1}^n |\text{l}_d(c_i)|$.
\end{defi}
Note that $\overline\ell$ corresponds to $\ell$ with the projection $\mathbb Z\to\mathbb Z/d\mathbb Z$ but with representatives in $]-\frac{d}{2},\frac{d}{2}]\cap\mathbb Z$ instead of $[0,d-1[\cap \mathbb Z$. Because, if we have $q^6=1$, the shortest way to write $q^2$ is $q\cdot q$ but to write $q^4$ we should use $q^{-1}\cdot q^{-1}$.
\begin{prop} The length of an element $g\in\overline G$ over $\overline S=\pi(S)$ is given by $\overline\ell$.
\end{prop}
\begin{proof}
The shortest way to write $q^k$ for $k\in\{0,\dots,d-1\}$ is using $q\cdot...\cdot q$ if $k\leq\frac{d}{2}$ and otherwise, as $q^d=1$, we write it as $q^{-1}\cdot...\cdot q^{-1}$.
\end{proof}
\begin{rmk}
We say that a word over $\overline S=\pi(S)$ is reduced in $\overline G$ if it has length $\overline\ell(w)$ when seen as an element of $\overline G$. 

For instance, if $d>2$ then $w=s(s*s)$ is reduced as it represents $2s$ which has $\ell=2$ but the word corresponding to $d_s=sT(s)\dots T^{d-1}(s)$ is not as it has $\overline\ell=0$. More generally, the word associated to $\sum\limits_{s\in S}g_ss$ is reduced in $\overline G$ if each $g_s$ is strictly less than $d$.
\end{rmk}
\subsection{Non-degeneracy}\label{nd}
Here we give a new proof of Rump's result on the non-degeneracy of finite cycle sets \cite{rump}, that is the fact that $s\mapsto s*s$ is bijective; in Dehornoy's paper it is used to obtain the bijectivity of $(s,t)\mapsto (s*t,t*s)$. Here we prove both of those results using the previous subsections, the proof of Rump's result has the advantage of being a simple direct consequence of the brace structure compared to the proof in \cite{rump} which involves several steps and constructions (such as using the so-called retraction of a cycle set).

Recall that we fix $(S,*)$ a finite cycle set of size $n$ with structure monoid (resp. group) $M$ (resp. $G$).
\begin{lem}[\cite{rump}]
\label{diagBij}
$\forall g,h\in G, \lambda^{-1}_g(g)=\lambda^{-1}_h(h)\Longleftrightarrow g=h$.

In particular, if $s,t\in S$ such that $s*s=t*t$ (i.e $\lambda^{-1}_s(s)=\lambda^{-1}_t(t)$, because $S$ is finite so we can invert the permutations $\psi(s)$), then $s=t$.
\end{lem}
\begin{proof}
If $g=h$ then trivially $\lambda^{-1}_g(g)=\lambda^{-1}_h(h)$. Suppose $\lambda_g(g)=\lambda_h(h)$, we want to show $gh^{-1}=1$. 

From Lemma \ref{ab-1} we find  : 
$$gh^{-1}=-\lambda_g(\lambda_{h^{-1}}(h))+g=-\lambda_g(\lambda_{g^{-1}}(g))+g=-\lambda_g(\lambda_{g}^{-1}(g))+g=-g+g=0=1.$$
\end{proof}
The following will be very useful to switch from working on the left to working on the right when looking at divisibility : 
\begin{prop}\label{o(T)}
\begin{enumerate}[label=(\roman*)]
\item The diagonal map $T:s\mapsto s*s$ is a bijection of $S$.
\item Let $o$ be the order of $T$. For any integer $k$ we have $ks=sT(s)\dots T^{o-1}(s)sT(s)\dots$ with exactly $k$ terms.
\item The order $o$ of $T$ divides $d$. In particular, for any integer $k$ and any $s$ in $S$, we have $\lambda^{-1}_{ks}(s)=T^{k}(s)$ and $kds=\left(sT(s)\dots T^{o-1}(s)\right)^k$.
\end{enumerate}
\end{prop}
\begin{proof}
As $S$ is finite and $T$ is injective by the previous lemma, it is bijective and so has finite order. The second point follows directly from an easy induction : $\lambda^{-1}_{(k+1)s}(s)=\lambda^{-1}_{ks+s}(s)=\lambda^{-1}_{\lambda_{ks}^{-1}(s)}(\lambda_{ks}^{-1}(s))=\lambda^{-1}_{T^{k}(s)}(T^{k}(s))=T^{k}(s)*T^{k}(s)=T^{k+1}(s)$ and then $(k+1)s=ks+s=ks\cdot \lambda_{ks}^{-1}(s)=ks\cdot T^{k}(s)=sT(s)\dots T^{k}(s)$  and, as $T^o(s)=s$, we regroup as many words $sT(s)\dots T^{o-1}(s)$ as possible. For the third  point, $T^d(s)=\lambda^{-1}_{ds}(s)=s$ as $ds\in\text{Soc}(G)$.
\end{proof}
\begin{cor} Let $G^t$ be the set of transposes of the elements of $G$ seen as matrices. Then $G^t$ is the structure group of a cycle set structure on $S^t$.
\end{cor}
Explicitly $\psi(s^t)=\psi^{-1}(T^{-1}(s))$.
\begin{proof}
First note that, because $G$ is generated by $S$, $G^t$ is generated by $S^t$. As $T$ is a bijection, for each $u$ the set $S^t$ contains exactly one element $s^t$ such that $D_{s^t}=D_u$, that is $s=T^{-1}(u)$. Moreover, as $G$ is permutation-free, so is $G^t$. So by Theorem \ref{condCS} it is the structure group of a cycle set $S^t$.
\end{proof}
\begin{ex}
Let $S=\{s_1,s_2,s_3\}$ with $\psi(s_1)=\psi(s_2)=\psi(s_3)=(123)=\sigma$. Then $s_i^t=(D_i P_\sigma)^t=P_\sigma^{-1}D_i=\pre{\sigma}{D_i}P_\sigma^2$, so for example $$s_2^t=\begin{pmatrix}0&1&0\\0&0&q\\1&0&0\end{pmatrix}^t=\begin{pmatrix}0&0&1\\1&0&0\\0&q&0\end{pmatrix}=D_3P_{(132)}=D_{\sigma(2)}P_\sigma^{-1}.$$
\end{ex}
In particular, this can be used to work on the columns in $G$: if we want an element of $G$ with the coefficient $q^{a_1},\dots,q^{a_n}$ read column by column, we can work in $G^t$, compute $\sum\limits_{i=1}^na_iT^{-1}(s_i)$ and transpose it to get the desired element in $G$.

Moreover, this also implies the easier characterization of divisibility in $G$ : 
\begin{cor}
\label{frac2}
Let $g,h$ be in $M$. Write $g=\sum\limits_{s\in S}g_ss$ and $h=\sum\limits_{s\in S}h_ss$ with $g_s,h_s\in\mathbb N$.

Then $g$ left-divides $h$ if and only if $g_s\leq h_s$ for all $s$.

Similarly, $g$ right-divides $h$ if and only if $g^t_{s^t}\leq h^t_{s^t}$ for all $s$, where $s^t=T^{-1}(s)$.
\end{cor}
\begin{cor}
\label{lcmgcd2}
Let $g,h$ be in $M$. Write $g=\sum\limits_{s\in S}g_ss$ and $g=\sum\limits_{s\in S}h_ss$ with $g_s,h_s\in\mathbb N$.

Then $g\wedge h=\sum\limits_{s\in S}\text{min}(g_s,h_s)s$ and $g\vee h=\sum\limits_{s\in S}\text{max}(g_s,h_s)s$.

Similarly $g\wedge_r h=\left(\sum\limits_{s\in S}\text{min}(g^t_{s^t},h^t_{s^t})s^t\right)^t$ and $g\vee_r h=\left(\sum\limits_{s\in S}\text{max}(g^t_{s^t},h^t_{s^t})s^t\right)^t$.
\end{cor}
\begin{prop}
For any $k\in \mathbb N$, $\psi(ks)(s)=T^k(s)$. In particular, the map $s\mapsto \psi(ks)(s)$ is a bijection of $S$. 
\end{prop}
\begin{proof}
This directly follows from Proposition \ref{o(T)} and the fact that $\lambda_g^{-1}=\psi(g)$ for any $g\in G$.
\end{proof}
\begin{cor}\label{-k} For any $k$ in $\mathbb N$ and $s$ in $S$, let $t=(T^k)^{-1}(s)$ then $-ks=(kt)^{-1}$.
\end{cor}
\begin{proof}
Let $t\in S$, we have $$(kt)^{-1}=(D_t^kP_{kt})^{-1}=P_{(kt)}^{-1}D_t^{-k}=\pre{\psi(kt)^{-1}}{D_t^{-k}}P_{kt}^{-1}=D_{\psi(kt)(t)}^{-k}P_{kt}^{-1}=D_{T^k(t)}^{-k}P_{kt}^{-1}.$$ Thus, if $t=(T^k)^{-1}(s)$, we find $D_{(kt)^{-1}}=D^{-k}_s$.

From Proposition \ref{o(T)} we have $t=\lambda_{kt}(T^k(t))$, so $$kt\cdot(-ks)=kt+\lambda_{kt}(-ks)=kt-k\lambda_{kt}(s)=kt-k\lambda_{kt}(T^k(t))=kt-kt=0=1.$$
\end{proof}
\begin{prop}[\cite{rcc}]
\label{bij}
The map $(s,t)\mapsto (s*t,t*s)$ is bijective.
\end{prop}
\begin{proof}
As $S$ is finite, so is $S\times S$, so we only have to show injectivity. Assume $s*t=s'*t'$ and $t*s=t'*s'$ for some $s,t,s',t'\in S$.
Then, from $s*t=s'*t'$ and $s+t=s(s*t)$, we have $\lambda_{s+t}\lambda^{-1}_{s'+t'}=\lambda_{s(s*t)}\lambda^{-1}_{s'(s'*t')}=\lambda_s\lambda_{s*t}\lambda^{-1}_{s'*t'}\lambda^{-1}_{s'}=\lambda_s\lambda^{-1}_{s'}$. Thus $\lambda_{s+t}\lambda^{-1}_{s'+t'}(s')=\lambda_s(s'*s')$.
As $s+t=t+s=t(t*s)$ we have by symmetry $\lambda_{s+t}\lambda^{-1}_{s'+t'}(t')=\lambda_t(t'*t')$. Thus $(s+t)(s'+t')^{-1}=-\lambda_{s+t}\lambda^{-1}_{s'+t'}(s'+t')+(s+t)=-\lambda_s(s'*s')-\lambda_t(t'*t')+(s+t)$

On the other hand, as $s+t=s(s*t)$, we have $(s+t)(s'+t')^{-1}=s(s*t)(s'*t')^{-1}s'^{-1}=ss'^{-1}=-\lambda_s\lambda_{s'}^{-1}(s')+s=-\lambda_s(s'*s')+s$.

Combining the above equalities gives $-\lambda_s(s'*s')+s=(s+t)(s'+t')^{-1}=-\lambda_s(s'*s')-\lambda_t(t'*t')+(s+t)$, thus we deduce $t-\lambda_t(t'*t')=0$, so $t*t=t'*t'$ and by the bijectivity of $T$ we find $t=t'$. By the same symmetry argument we obtain $s=s'$, and this concludes the proof.
\end{proof} 
\section{Bounding of the Dehornoy's class}\label{coxlike}
In this section our goal is to study the Dehornoy's class of the solution and some consequences on the germs. The main interest is the use of the monomial approach that provides a nice combinatorial tool that was used for computer analysis of cycle sets that were enumerated up to size 10 by Akgün, Mereb and Vendramin in \cite{enumeration}.

We fix a finite cycle set $(S,*)$ of size $n$  with structure monoid (resp. group) $M$ (resp. $G$), of Dehornoy's class $d>1$ and associated germ $\overline G$.
\begin{defi} The permutation group $\mathcal G_S$ associated to a cycle set $S$ is the subgroup of $\mathfrak S_n$ generated by $\psi(s_i), 1\leq i\leq n$.

When the context is clear we will simply write $\mathcal G$.
\end{defi}
$\mathcal G$ is precisely the image of the map sending $g$ in $G$ to $P_g$. Note that, as $P_\sigma P_\tau=P_{\tau\sigma}$, we have that $\psi(gh)=\psi(h)\psi(h)$, thus an antimorphism. This won't pose problem here, as we'll only use $\psi(s^n)=\psi(s)^n$. Equivalently, it is the image of the morphism given by the restriction $\lambda|^S\colon (G,\cdot)\to \text{Aut}(S)$.
\begin{defi}[\cite{cycle}] A subset $X$ of $S$ is said to be $\mathcal G$-invariant if for every $s\in S$, $\psi(s)(X)\subset X$.
$S$ is called decomposable if there exists a proper partition $S=X\sqcup Y$ such that $X,Y$ are $\mathcal G_S$-invariant.

In this case $(X,*_{|_X})$ and $(Y,*_{|_Y})$ are also cycle sets.

A cycle set that is not decomposable is called indecomposable.
\end{defi}
\begin{ex}
For $S=\{s_1,s_2,s_3,s_4\}$ and $\psi(s_1)=\psi(s_2)=(2143)$, $\psi(s_3)=\psi(s_4)=(2134)$, we have $\mathcal G=\langle (2143),(2134)\rangle<\Sym_n$. We see that $X=\{s_1,s_2\}$ and $Y=\{s_3,s_4\}$ are both $\mathcal G$-invariant and their respective cycle set structure are given by $\psi_X(s_1)=\psi_X(s_2)=(12)$ and $\psi_Y(s_3)=\psi_Y(s_4)=(34)$.
\end{ex}
In personal communications \cite{raul}, the following conjecture was mentioned: 
\begin{conj}[\cite{raul}]
\label{conj2}
 If $S$ is indecomposable then $d\leq n$.
\end{conj}
Note that, taking $S=\{s_1,\dots,s_n\}$ with $\psi(s)=(12\dots n)$ for all $s$ provides an indecomposable cycle set that attains this bound.
\par Using a python program based on the proof of Proposition \ref{class} and the enumeration from \cite{enumeration}, we find the following maximum values of the class of cycle sets of size $n$:

\begin{center}
\begin{tabular}{c|cccccccccc}
$n$&1&2&3&4&5&6&7&8&9&10\\\hline
$d_{\text{max}}$&1&2&3&4&6&8&12&15&24&30
\end{tabular}
\end{center}
This corresponds to the OEIS sequence \href{https://oeis.org/A034893}{A034893} \textit{"Maximum of different products of partitions of n into distinct parts"}, studied in \cite{part} where the following is proved:
\begin{lem}[\cite{part}] Let $n\geq 2$ be written as $n=\mathcal T_m+l$ where $\mathcal T_m$ is the biggest triangular number ($\mathcal T_m=1+2+\dots+m$) with $\mathcal T_m\leq n$ (and so $l\leq m$). Then the maximum value $$a_n=\max\left(\left\{\prod\limits_{i=1}^k n_i\middle|k\in\mathbb N, 1\leq n_1<\dots<n_k, n_1+\dots+n_k=n\right\}\right)$$ is given by  $$a_n=a_{\mathcal T_m+l}=\begin{cases}\frac{(m+1)!}{m-l},&0\leq l\leq m-2\\\frac{m+2}{2}m!,&l=m-1\\(m+1)!,&l=m.\end{cases}$$
\end{lem}
This leads to the following conjecture:
\begin{conj}\label{conj1} The class $d$ of $S$ is bounded above by $a_n$ and the bound is minimal.
\end{conj}
As a consequence we obtain the following result: 
\begin{prop}[\cite{brace}]\label{ddivG} The class $d$ divides the order of $\mathcal G$. In particular $d$ divides $n!$.
\end{prop}
\begin{proof}
For $s\in S$, the set $\{ks\mid k\in\mathbb Z\}$ is a subgroup of $(G,+)$, and the smallest integer $d_s$ such that $d_ss$ is in the socle corresponds to the order of $\psi(s)$ in $(\mathcal G,+)$, which thus divides $|\mathcal G|$. Thus the lcm of the $d_s$ also divides the order of $\mathcal G$.

As $d$ is the lcm of all the $d_s,s\in S$, it also divides $|\mathcal G|$.
\end{proof}
The landau function $g\colon\mathbb N^*\to\mathbb N^*$ (\cite{landau}) is defined as the largest order of a permutation in $\mathfrak S_n$. 
\begin{prop}\label{conj-proof} If $S$ is square-free and $\mathcal G$ abelian then $d\leq a_n$
\end{prop}
That is, under these conditions the bound part of Conjecture \ref{conj1} holds.
\begin{proof} If $S$ is square-free, then for all $s\in S$ we have by definition $T(s)=s$, so for any $k\in\mathbb Z$, $ks=sT(s)\dots T^{k-1}(s)=s^k$ so $\{ks\mid k\in\mathbb Z\}$ is a subgroup of $(G,\cdot)$ and the smallest integer $d_s$ such that $d_ss$ is in the socle corresponds to the order of $\psi(s)$ in $(\mathcal G,\cdot)$, which thus divides $e(\mathcal G)$ the exponent of $\mathcal G$ (the lcm of the orders of every element). So $d$ will also divide $e(\mathcal G)$.

As $\mathcal G$ is abelian and finite, there exists an element with order equal to its exponent, so the exponent is bounded by the maximal order of an element, i.e. $d\mid e(\mathcal G)\leq g(n).$

By the decomposition in disjoint cycles, $g(n)$ is equal to the maximum of the lcm of partitions of $n$: $$g(n)=\max\left(\left\{\text{lcm}(n_1,\dots,n_k)\middle|k\in\mathbb N, 1\leq n_1\leq\dots\leq n_k, n_1+\dots+n_k=n\right\}\right)$$ 
Moreover, by properties of the lcm, if $1\leq n_i=n_j$, as $\text{lcm}(n_i,n_j)=n_i$, the max is unchanged by replacing $n_j$ by only $1$'s. And as the lcm of a set is bounded above by the product of the elements, we have $g(n)\leq a_n$. Thus $d\leq g(n)\leq a_n$.
\end{proof}
\begin{prop} The followings hold:
\begin{enumerate}[label=(\roman*)]
\item $\psi\colon G\rightarrow\mathcal G$ factorizes through the projection $G\to\overline G$
\item We have the following divisibilities: \begin{itemize}\item $o(T)\mid d$ \item $d\mid \#\mathcal G$\item $\#\mathcal G\mid d^n$  \end{itemize}
\end{enumerate}
where $o(T)$ is the order of the diagonal permutation $T$, $\#\mathcal G$ denotes its order $|\mathcal G|$ (to avoid confusion with $\mid$ for divisibility).
\end{prop}
\begin{proof}
(i) follows from the definition of $d$ as $\psi(ds)=\text{id}$.

For (ii), the first divisibility is  Proposition \ref{o(T)}, the second is Proposition \ref{ddivG} and the third is (i).
\end{proof}
For a positive integer $k$, denote by $\pi(k)$ the set of divisors of $k$.
\begin{cor} \label{pp}
We have $\pi(d)=\pi(\#\mathcal G)$.

In particular, $d$ is a prime power iff $\#\mathcal G$ is a prime power.
\end{cor} 
This means that our later results, which will involve the condition "$d$ is a prime power" can also be restated for $\#\mathcal G$.
\begin{proof}
As $d$ divide $\#\mathcal G$ any divisor of $d$ is a divisor of $\#\mathcal G$. Conversely, if $p$ is a prime divisor of $\#\mathcal G$ then it divides $d^n$ and thus divides $d$.
\end{proof}
\begin{lem}\label{indN} If $S$ is indecomposable then $n$ divides $\#\mathcal G$.

In particular, $\pi(n)\subseteq \pi(\#\mathcal G)=\pi(d)$, and thus if $d$ is a prime power then $n$ is also a power of the same prime.
\end{lem}
\begin{proof}
By \cite{etingof} $S$ is indecomposable iff $\mathcal G$ acts transitively on $S$. By the orbit stabilizer theorem, for any $s$ in $S$ we have $\# \text{Orb}(x)=\frac{\#\mathcal G}{\#\text{Stab}(x)}$. So if $S$ is indecomposable there is a unique orbit of size $n$ so $n$ divides $\#\mathcal G$. The last statements are a direct consequence of this divisibility and the previous corollary.
\end{proof}
\begin{lem}
If $S$ is indecomposable and $\mathcal G$ is abelian, then $n=|\mathcal G|$
\end{lem}
\begin{proof}\footnote{https://math.stackexchange.com/a/1316138}
Again by \cite{etingof} S is indecomposable iff $\mathcal G$ acts transitively on $S$. Let $x_0\in S$, by transitivity for all $x\in S$, there exists $\sigma\in\mathcal G$ such that $x=\sigma(x_0)$. Let $\tau\in\mathcal G$ be such that we also have $x=\tau(x_0)$, we will show that $\tau=\sigma$. For all $y\in S$, there exists $\nu\in\mathcal G$ such that $y=\nu(x)$, thus $\sigma(y)=\sigma(\nu(x))=\sigma(\nu(\tau(x_0))=\tau(\nu(\sigma(x_0))=\tau(y)$. So an element of $\mathcal G$ is uniquely determined by its image of $x_0$, thus $|S|\geq |\mathcal G|$, and the other inequality follows by transitivity.
\end{proof}
Let $k\geq 1$ and $kG$ be the subgroup of $G$ generated by $kS=\{ks\mid s\in S\}$.
The following result was simultaneously introduced in \cite{lebed}:
\begin{prop}
\label{Sk}
For $\geq 1$, $kG$ induces a cycle set structure on $kS$.
\end{prop}
Explicitly, $\psi(ks)(kt)=k\lambda_{ks}^{-1}(t)$.
\begin{proof}
First note that $kG$ is a left-ideal of $G$, in particular a subbrace. Define $ks\star ts=\psi(ks)(t)=k\lambda_{ks}^{-1}(t)=\lambda_{ks}^{-1}(kt)$. We want to show that $(ks\star kt)\star(ks\star ku)=(kt\star ks)\star(kt\star ks)$.

We have $(ks\star kt)\star(ks\star ku)=\lambda_{ks}^{-1}(kt)\star \lambda_{ks}^{-1}(ku)=\lambda_{\lambda_{ks}^{-1}(kt)}^{-1}(\lambda_{ks}^{-1}(ku))$.

The conclusion then follows from Lemma \ref{ybe} $$(ks\star kt)\star(ks\star ku)=\lambda_{\lambda_{ks}^{-1}(kt)}^{-1}(\lambda_{ks}^{-1}(ku))=\lambda_{\lambda_{kt}^{-1}(ks)}^{-1}(\lambda_{kt}^{-1}(ku))=(kt\star ks)\star(kt\star ku).$$
\end{proof}
\begin{prop}
Let $k$ be a positive integer smaller than $d$, then $(kS,\star)$ is of class $\frac{d}{\text{gcd}(d,k)}$.

Moreover, $((d+1)S,\star)$ is the same, as a cycle set, as $(S,*)$.
\end{prop}
This means that this construction provides, at most, $d$ different cycle sets.
\begin{proof} Recall that $j(ks)=(jk)s$. Thus $a(ks)$ is in the socle when $ak$ is a multiple of $d$, so we deduce that $kS$ is of class $\frac{\text{lcm}(d,k)}{k}=\frac{d}{\text{gcd}(d,k)}$.

By definition of $d$, we have that $(dS,\star)$ is the trivial cycle set ($\psi(s)=\text{id}$), thus $\psi((d+1)S)=\psi(s)$.
\end{proof}
\section{Sylow subgroups and decomposition}
Recall that for $k>1$, $\M_n^k$ denotes the group of monomial matrices with non-zero coefficients powers of $\zeta_k$, and $\iota_k^{kl}$ is the embedding $\M_n^k\hookrightarrow\M_n^{kl}$ sending $\zeta_k$ to $\zeta_{kl}^l$. Given two subgroups $H,K<G$, their internal product subset is defined by $HK=\{hk\mid h\in H,k\in K\}$. If $H$ and $K$ have trivial intersection and $HK=KH$, the set product $HK$ has a natural group structure called the Zappa--Szép product of $H$ and $K$. We apply this to the Sylow-subgroups of the germs to obtain that any finite cycle set can be constructed from the Zappa--Szép product of the germs of cycle sets of class a prime power.
\begin{defi}\label{intprod}
Let $k,l$ be integers such that $k,l>1$. Let $m$ be a common multiple of $k$ and $l$, with $m=ka=lb$ for some $a,b\geq 1$. Given two subgroups $G<\M_n^k$, $H<\M_n^l$ by $G\bowtie_m H$ we denote the subset $\iota_k^m(G)\iota_l^m(H)$ of $\M_n^m$.

Identifying $G$ and $H$ with their image in $\M_n^m$, we say that they commute (\cite{group}) if $GH=HG$ as sets, i.e. for any $(g,h)$ in $G\times H$, there exists a unique $(g',h')$ in $(G\times H)$ such that $gh=h'g'$.
\end{defi}
\begin{rmk}
This operation can be thought of as taking elements of $G$ and $H$, changing appropriately the roots of unity (with $\zeta_k=\zeta_m^a$ and $\zeta_l=\zeta_m^b$) and taking every  product of such elements (we embed $G$ and $H$ in $\M_n^m$ and take their product as subsets).

When $k$ and $l$ are coprime, $G$ and $H$ can be seen as subgroups of $\M_n^m$ with trivial intersection, and so if they commute we have that $G\bowtie_m H$ is a group called the Zappa--Szép product of $G$ and $H$ (\cite{group}, Product Theorem).
\end{rmk}
Let $(S,*_1),(S,*_2)$ be two cycle sets, over the same set $S$, of coprime respective classes $d_1,d_2$ and germs $\overline G_1,\overline G_2$. Let $d=d_1d_2$ and $\overline G=\overline G_1\bowtie_d \overline G_2$ (which, in general, is only a subset of $\M_n^d$), and we identify each $\overline G_i$ with its image in $\overline G$.
\begin{defi}
$S_1$ and $S_2$ are said to be $\bowtie$-compatible if $\overline G$ is a germ of the structure group of some cycle set which we'll denote $S_1\bowtie S_2$.
\end{defi}
We now construct a candidate $S_1\bowtie_d S_2$ for which $\overline G$ could be the germ. This candidate is not, in general a cycle set, but if it is, its class is a divisor of $d$. Then we will state the condition for it to be a cycle set.

For clarity, we will put a subscript to distinguish between the respective structures of $S_1$ and $S_2$: $\psi_1(s)$ will denote the permutation given by $*_1$.

\begin{algorithm}[H]
\caption{Constructing $S_1\bowtie_d S_2$}
\label{algdia}
\textbf{Input:} A set $S$ with two cycle sets structure $*_1,*_2$ on $S$ of coprime classes $d_1,d_2$\hfill \mbox{} \\
\textbf{Output:} A couple $(S,*)$ with $*$ a binary operation \hfill\mbox{}
\begin{algorithmic}[1]
\State Compute $(u,v)$ the solution to Bézout's identity $d_2u+d_1v=1[d]$
\For{$i=1$ to $n$}
\State Compute $g=us_i\in\overline G_1$
\State Let $\sigma=\psi_1(us_i)$
\State Compute $h=vs_{\sigma(i)}=v\lambda_{g}^{-1}(s_i)\in\overline G_2$
\State Let $\psi(s_i)$ be the permutation of $\iota_{d_1}^d(g_1)\iota_{d_2}^d(g_2)$
\EndFor
\State \Return $S_1\bowtie_d S_2=(S,*)$ with $s_i*s_j=s_{\psi(s_i)(j)}$.
\end{algorithmic}
\end{algorithm}
\begin{rmk}
The heart of the algorithm is line $5$ which relies on $ks\cdot kt=ks+k\lambda_{ks}(t)$.

To obtain an element with diagonal part $D_{s_i}$, we have to take $t=s_{\sigma(i)}=\lambda_{ks}^{-1}(s_i)$ with here $\sigma=\psi(ks_i)$ and as we apply $\iota^d$ on the elements (in $S_1$ this does $q\mapsto q^{d_2}$ and in $S_2$ $q\mapsto q^{d_1}$), we obtain $D_{s_i}=D_i^{d_2u+d_1v}=D_i$ from lign 1.
\end{rmk}
\begin{ex}
Take two cycle sets of size $n=5$ and class respectively $2$ and $3$, and apply Algorithm \ref{algdia} providing a candidate for a cycle set of class $6$:

Let $S_1=\{s_1',\dots,s_5'\}$ and $S_2=\{s_1'',\dots,s_5''\}$, with $(S_1,\psi_1),(S_2,\psi_2)$ given by: 
\begin{align*}
\psi_1(s_1')=\psi_1(s_3')=(1234)&&\psi_1(s_2')=\psi_1(s_4')=(1432)&&\psi(s_5')=\text{id}\\
\psi_2(s_1'')=\psi_2(s_2'')=(354)&&\psi_2(s_3'')=\psi_2(s_4'')=\psi_2(s_5'')=(345)&&
\end{align*}
Where $S_1$ is of class $d_1=2$ and $S_2$ of class $d_2=3$.

Consider their respective germs $\overline G_1$ and $\overline G_2$ of order $2^5$ and $3^5$. Then define $\overline G=\overline G_1\bowtie_6\overline G_2$ over the basis $S=\{s_1,\dots,s_5\}$. For instance: $$\iota_2^6(s_1')=\iota_ 2^6\left(\begin{pmatrix}
0&\zeta_2&0&0&0\\0&0&1&0&0\\0&0&0&1&0\\1&0&0&0&0\\0&0&0&0&1
\end{pmatrix}\right)=\begin{pmatrix}
0&\zeta_6^3&0&0&0\\0&0&1&0&0\\0&0&0&1&0\\1&0&0&0&0\\0&0&0&0&1
\end{pmatrix}$$
$$
\iota_3^6(s_1'')=\iota_ 3^6\left(\begin{pmatrix}
1&0&0&0&0\\0&1&0&0&0\\0&0&0&0&\zeta_3\\0&0&1&0&0\\0&0&0&1&0\end{pmatrix}\right)=\begin{pmatrix}
1&0&0&0&0\\0&1&0&0&0\\0&0&0&0&\zeta_6^2\\0&0&1&0&0\\0&0&0&1&0\end{pmatrix}$$
To construct an element $g\in\overline G$ with $D_g=D_{s_3}$ we first solve Bézout's identity modulo 6: $3u+2v=1[6]$, a solution is given by $u=1$ and $v=2$, so we will multiply some $\iota_2^6(s_i')$ and $\iota_2^6(2s_j'')$ so that their product has diagonal part $D_{s_3}^1(D_{s_3}^2)^2=D_{s_3} \text{mod} q^6$. Recall that: $$ks\cdot kt=ks+k\lambda_{ks}(t).$$
Here we want $s=\lambda_{ks}(t)=s_3$, $k=3u$ and $l=2v$, so we take $s=3$. As $\sigma=\psi(1s_3')=\psi(s_3')=(1234)$, we have $t=s_{\sigma(3)}''=s_4''$, and note that $2s_4''=s_4''s_5''$. Finally: 
\begin{align*}\iota_2^6(s_3')\iota_3^6(2s_4'')&=\begin{pmatrix}
0&1&0&0&0\\0&0&1&0&0\\0&0&0&\zeta_6^{3\cdot 1}&0\\1&0&0&0&0\\0&0&0&0&1
\end{pmatrix}
\begin{pmatrix}
1&0&0&0&0\\0&1&0&0&0\\0&0&0&0&1\\0&0&\zeta_6^{2\cdot 2}&0&0\\0&0&0&1&0
\end{pmatrix}\\&
=\begin{pmatrix}
0&1&0&0&0\\0&0&0&0&1\\0&0&\zeta_6^{3+4}&0&0\\1&0&0&0&0\\0&0&0&1&0
\end{pmatrix}
=\begin{pmatrix}
0&1&0&0&0\\0&0&0&0&1\\0&0&\zeta_6&0&0\\1&0&0&0&0\\0&0&0&1&0
\end{pmatrix}\end{align*}
This will be our candidate for $s_3$. Doing this for all the generators we find: 
$$\psi(s_1)=(124)(35),\psi(s_2)=(1532),\psi(s_3)=(1254),\psi(s_4)=(132)(45),\psi(s_5)=(354).$$
Unfortunately, this isn't a cycle set: $(s_1*s_2)*(s_1*s_1)=s_4*s_2=s_1$ whereas $(s_2*s_1)*(s_2*s_1)=s_5*s_5=s_4$. This also means that $\overline G$ is not a brace, as for instance $s_1+s_2\neq s_2+s_1$.
\end{ex}
\begin{prop}
If $\overline G_1$ and $\overline G_2$ commute, then $S_1$ and $S_2$ are $\bowtie$-compatible. 
\end{prop}
In this case, $\overline G=\overline G_1\bowtie_d \overline G_2$ is the Zappa--Szép product of $\overline G_1$ and $\overline G_2$.
\begin{proof}
As $d_1$ and $d_2$ are coprime, it is clear that $\overline G_1\cap\overline G_2=\{1\}$. 

By (\cite{group}, Product Theorem), $\overline G$ is a subgroup of $\M_n^k$ if and only if $\overline G_1$ and $\overline G_2$ commute, i.e. $\overline G=\overline G_1\bowtie_d \overline G_2=\overline G_2\bowtie_d \overline G_1$. As $\overline G_1$ and $\overline G_2$ have different (non-zero) coefficient-powers, a product $g_1g_2$ of two non-trivial elements from $\iota_{d_1}^d(\overline G_1)$ and $\iota_{d_2}^d(\overline G_2)$ cannot be a permutation matrix.

With Algorithm \ref{algdia}, we can construct the generators for $G<\Sigma_n$ such that $G/\langle ds\rangle_{s\in S}=\overline G$ (i.e $G$ is obtained by forgetting $q^d=1$ in $\overline G$). Then $G$ satisfies condition (i) of Theorem \ref{condCS}, this finishes proof.
\end{proof}
\begin{rmk}
To check whether $\overline G_1$ and $\overline G_2$ commute, we can restrict to the generators and check that: $$\forall s\in S_1, t\in S_2,\exists s'\in S_1,t'\in S_2\text{ such that } st=t's'.$$
\end{rmk}
\begin{prop}
If $S_1$ and $S_2$ satisfy the following "mixed" cycle set equation 
\begin{equation}\label{zsc}
\forall s,t,u\in S, (s*_1t)*_2(s*_1u)=(t*_2s)*_2(t*_2u)
\end{equation}
then $S_1$ and $S_2$ are $\bowtie-$compatible and $(S=S_1\bowtie_d S_2,*)$ is a cycle set.
\end{prop}
Explicitly, from Algorithm \ref{algdia}, we have $\psi(s_i)=\psi_2\left(vs_{\psi_1(us_i')(i)}''\right)\circ\psi_1\left(us_i'\right)$ with $u,v$ such that $d_2u+d_1v=1[d_1d_2]$.
\begin{proof}
Here we will work over monomial matrices as we do not yet have a Brace structure on $\overline G$.

We will use the previous proposition and show how Equation \ref{zsc} naturally arises from considering the commutativity of the germs. For clarity, although our two cycle sets have the same underlying set $S=\{s_1,\dots,s_n\}$, we will distinguish where we see those elements by writing $s'$ for $(S,*_1)$ and $s''$ for $(S,*_2)$.

Let $s_i'\in S_1,s_j''\in S_2$, then in $\overline G$: $$s_i's_j''=D_i^{d_2}P_{s_i'}D_j^{d_1}P_{s_j''}=D_i^{d_2}D_{\psi_1(s_i')^{-1}(j)}^{d_1}P_{s_i'}P_{s_j''}.$$ We want some $s_k'\in S_1,s_l''\in S_2$ such that $s_i's_j''=s_l''s_k'$, i.e: $$D_i^{d_2}D_{\psi_1(s_i')^{-1}(j)}^{d_1}P_{s_i'}P_{s_j''}=D_l^{d_1}D_{\psi_2(s_l'')^{-1}(k)}^{d_2}P_{s_l''}P_{s_k'}.$$
As $d_1$ and $d_2$ are coprime, they're in particular different, so we must have: $$\begin{cases}
D_i^{d_2}=D_{\psi_2(s_l'')^{-1}(k)}^{d_2}\\D_{\psi_1(s_i')^{-1}(j)}^{d_1}=D_l^{d_1}\\P_{s_i'}P_{s_j''}=P_{s_l''}P_{s_k'}.
\end{cases}$$
From which we first deduce: $k=\psi_2(s_l'')(i)$ and $j=\psi_1(s_i')(l)$, or equivalently $s_k=s_l*_2 s_i$ and $s_j=s_i*_1 s_l$. So taking this $k$ and $l$ we get $D_{s_i's_j''}=D_{s_l''s_k'}$. We are left with last of the three conditions, which then becomes: $$P_{s_i'}P_{s_i'*_1 s_l''}=P_{s_l''}P_{s_l''*_2 s_i'}.$$ As $P_\sigma P_\tau=P_{\tau\sigma}$, the last condition is equivalent to $$\psi_2(s_i'*_1s_l'')\circ\psi_1(s_i')=\psi_1(s_l''*_2s_i')\circ\psi_2(s_l'').$$ As $s_l''\in S_2$, $\psi_2(s_i'*_1s_l'')$ is seen as the action of an element of $S_2$, so all this becomes equivalent to: $$\forall s,t,u\in S, (s*_1t)*_2(s*_1u)=(t*_2s)*_2(t*_2u).$$
\end{proof}
\begin{rmk} The condition that the classes are coprime is used, with Bézout's identity, to have generators of the group $\overline G$ (elements with diagonal part $D_i$). Otherwise, say for instance that the classes are powers of the same prime, $d_1=p^a$ and $d_2=p^b$ with $b\leq a$. Then $\iota_{d_2}^d$ is the identity and $\iota_{d_1}^d$ will add elements with higher coefficient powers (or equal), thus we do not get any new generators (or too many in the case $a=b$).
\end{rmk}
We've seen how to construct cycle sets from ones of the same size and coprime classes. Now we show that this is enough to get all cycle sets from just ones of prime-power class:

Let $d=p_1^{a_1}\dots p_r^{a_r}$ be the prime decomposition of $p$ ($a_i>0$ and $p_i\neq p_j$), and write $\alpha_i=p_i^{a_i}$ for simplicity. We use techniques inspired by \cite{primitive} to construct new cycle sets from two with coprime Dehornoy's class.

Fix again a cycle set $S$ of size $n$ and class $d>1$, with germ $\overline G$.
By Proposition \ref{Sk}, given $k>0$ diving $d$, the subgroup $\overline kG$ generated by $kS=\{ks\mid s\in S\}$ is the germ of a structure group, and has for elements the matrices whose coefficient-powers are multiples of $k$.
\begin{lem}\label{sylow} Let $\beta_i=\frac{d}{\alpha_i}$ then 
\begin{enumerate}[label=(\roman*)]
\item For each $i$, $\beta_i\overline G$ is a $p_i$-Sylow of $\overline G$.
\item Two such subgroups commute (i.e. $\beta_i\overline G\cdot \beta_j\overline G=\beta_j\overline G\cdot \beta_i\overline G$).
\item $\overline G$ is the product of all those subgroups.
\end{enumerate}
\end{lem}
\begin{proof}
(i) follows directly from the fact that that $\beta_i\overline G$ is a left ideal of $\overline G$ and has order $\alpha_i^n$.

(ii) and (iii) follows from Remark \ref{prodsum} and again the fact that $\beta_i\overline G$ is a left ideal of $\overline G$, and with corresponding cardinality.
\end{proof}

\begin{ex}\label{exdec}
The first example where $S$ is indecomposable but has class product of different primes is $n=8,d=6$ given by: 
\begin{flalign*}
\psi(s_1)=(12)(36)(47)(58),\quad\quad&
\psi(s_2)=(1658)(2347),\\
\psi(s_3)=(1834)(2765),\quad\quad&
\psi(s_4)=(12)(38)(45)(67),\\
\psi(s_5)=(1438)(2567),\quad\quad&
\psi(s_6)=(1856)(2743),\\
\psi(s_7)=(16)(23)(45)(78),\quad\quad&
\psi(s_8)=(14)(25)(36)(78)\end{flalign*}
Here, $\overline G$ decomposes as the Zappa--Szép product $3\overline G\bowtie_62\overline G$ of its 2- and 3- Sylow. If we denote by $(S_2,\psi_2)$ and $(S_3,\psi_3)$ their respective cycle set structure then we find: 
\begin{align*}
\psi_2(s_1')=\psi_2(s_2')&=(1476)(2583),\\
\psi_2(s_3')=\psi_2(s_6')&=(18)(27)(36)(45),\\
\psi_2(s_4')=\psi_2(s_5')&=(1674)(2385),\\
\psi_2(s_7')=\psi_2(s_8')&=(12)(34)(56)(78)
\end{align*}
and 
\begin{align*}
\psi_3(s_1'')=\psi_3(s_3'')=\psi_3(s_5'')=\psi_3(s_7'')&=(135)(264),\\
\psi_3(s_2'')=\psi_3(s_4'')=\psi_3(s_6'')=\psi_3(s_8'')&=(153)(246).
\end{align*}
\end{ex}
Lemma \ref{sylow} can be rephrased as $\overline G=\beta_1\overline  G\bowtie_d\dots\bowtie_d\beta_r\overline  G$. As the germ can be used to reconstruct the structure group and thus the cycle set, the following theorem summarizes these results from an enumeration perspective, that is constructing all solutions of a given size. 
\begin{thm}\label{sylows}
Any finite cycle set can be constructed from the Zappa--Szép product of the germs of cycle sets of class a prime power.
\end{thm}
\begin{proof}
Any cycle set is determined by its structure monoid, which can be recovered from the germ. By Lemma \ref{sylow} and the above construction, the germ can be decomposed and reconstructed from its Sylows, which also determine cycle sets by Proposition \ref{Sk}.
\end{proof}
\begin{rmk}
The class of the cycle set constructed will Algorithm \ref{algdia} will, in general, only be a divisor of the product of the prime-powers. This happens because nothing ensures that, for instance, the cycle set obtained is not trivial: we only know that $d_1d_2s$ is diagonal, but it is not necessarily minimal.
\end{rmk}
\begin{rmk}
This construction is similar to the matched product of braces $B_1\bowtie B_2$ appearing in \cite{matched,matched2}. Key differences are that we directly construct a cycle set with permutation group $B_1\bowtie B_2$ (whereas the authors of \cite{matched2} construct one over the set $B_1\bowtie B_2$) and that our construction doesn't rely on groups of automorphisms thanks to the natural embedding $\iota_k^{kl}$. Moreover, instead of classifying all braces, the existence of the germs suggests it is enough to classify braces with abelian group $(\mathbb Z/d\mathbb Z)^n$ for all $d$ and $n$ to recover all cycle sets.
\end{rmk}
\begin{cor}
Any cycle set is induced (in the sense of using the decomposability and Zappa--Szép product) by indecomposable cycle sets of smaller size and class, both powers of the same prime.
\end{cor}
\begin{proof}
Let $S$ be obtained from the germ as an internal product of $S_1,\dots,S_r$ of classes respectively $p_1^{a_1},\dots,p_r^{a_r}$ with distinct primes. Then, consider a decomposition of each $S_i$ as indecomposable cycle sets: so up to a change of enumeration, the matrices in the structure group of $S_i$ are diagonal-by-block with each block corresponding to a cycle set, so with class dividing the class $p_i^{a_i}$ of $S_i$, thus also a power of $p_i$. By Lemma \ref{indN}, the size of those indecomposable cycle sets must also be powers of $p_i$.
\end{proof}
However, as far as the author knows, there is no "nice" way, given two cycle sets, to construct all cycle sets that decompose on those two, thus the above result is an existence result but not a constructive one, unlike the Zappa--Szép product previously used.
\begin{rmk}
Starting from a cycle set, we first write it as a Zappa--Szép product of its Sylows and then decompose each Sylow-subgroup if the associated cycle set is decomposable. If one proceeds the other way, first decomposing and then looking at the Sylows of each cycle set of the decomposition, we obtain less information. For instance, if $S=\{s_1,\dots,s_6\}$ with $\psi(s_i)=(1\dots 6)$ for all $i$, then $S$ is not decomposable, but the cycle sets obtained from its Sylows $2S$ and $3S$ are decomposable ($\psi_2(s_i)=(14)(25)(36)$ and $\psi_3(s_i)=(135)(246)$ for all $i$, having respectively 3 and 2 orbits). 
\end{rmk}
\begin{ex} In Example \ref{exdec}, $3$ does not divide $8$ so $S_3$ has to be decomposable, and indeed it decomposes as $S_3=\{s_1'',s_3'',s_5''\}\sqcup\{s_2'',s_4'',s_6''\}\sqcup\{s_7'',s_8''\}$.
\end{ex}
\begin{cor}
Let $N(n,d)$ be the number of cycle sets of size $n$ and of class a divisor of $d=p_1^{a_1}\dots p_r^{a_r}$. Then we have: $N(n,d)\leq \prod_i N(n,p_i^{a_i})$.
\end{cor} 
For $n=10$, we find that there is approximately 67\% of cycle sets that have class a prime-power ($\sim$ 3.3 out of $\sim$ 4.9 millions). We hope that this number greatly reduces as $n$ increases (as hinted by the previous values, for $n=4$ it is $99\%$), as more values of $d$ are possible (Conjecture \ref{conj1}).
\emergencystretch=1em
\printbibliography
\end{document}